\documentclass[12pt,a4paper,reqno]{amsart}

\usepackage{stix} % stix fonts
\usepackage[cal=boondoxo]{mathalfa} % mathcal from STIX, unslanted a bit

\usepackage[english]{babel}
\usepackage[utf8]{inputenc}
\usepackage{amssymb}
\usepackage[font=small,labelfont=bf]{caption}
\usepackage{tabu}
\usepackage{pbox}
\usepackage{multirow}
\usepackage[headings]{fullpage}
\usepackage{xcolor}
\usepackage[shadow,color=pink]{todonotes}
\usepackage[colorlinks]{hyperref}

\DeclareMathOperator{\ad}{ad}
\DeclareMathOperator{\Aut}{Aut}
\DeclareMathOperator{\Diff}{Diff}
\DeclareMathOperator{\E}{E}
\DeclareMathOperator{\GL}{GL}
\DeclareMathOperator{\I}{Isom}
\DeclareMathOperator{\OO}{O}
\DeclareMathOperator{\Ric}{Ric}
\DeclareMathOperator{\scal}{scal}
\DeclareMathOperator{\SO}{SO}
\DeclareMathOperator{\Sym}{Sym}
\DeclareMathOperator{\tr}{tr}

\newcommand{\so}{\mathfrak{so}}

\theoremstyle{plain}
\newtheorem{theorem}{Theorem}[section]

\theoremstyle{remark}
\newtheorem{remark}[theorem]{Remark}

\numberwithin{equation}{section}

\begin{document}

\title{Isometry groups of three-dimensional Lie groups}

\author{Ana Cosgaya}
\address{Universidad Nacional de Rosario, EFB-FCEIA,
  Departamento de Ma\-te\-má\-ti\-ca. Av. Pellegrini 250, 2000
  Rosario, Argentina.}
\email{\href{mailto:acosgaya@fceia.unr.edu.ar}{acosgaya@fceia.unr.edu.ar}}

\author{Silvio Reggiani}
\address{CONICET and Universidad Nacional de Rosario, ECEN-FCEIA,
  Departamento de Ma\-te\-má\-ti\-ca. Av. Pellegrini 250, 2000
  Rosario, Argentina.}
\email{\href{mailto:reggiani@fceia.unr.edu.ar}{reggiani@fceia.unr.edu.ar}}
\urladdr{\url{http://www.fceia.unr.edu.ar/~reggiani}}

\date{\today}

\thanks{Supported by CONICET. Partially supported by SeCyT-UNR and ANPCyT}



\maketitle

\begin{abstract}
  We compute the full isometry group of any left invariant metric on a
  simply connected, non-unimodular Lie group of dimension three. As an
  application, we determine the index of symmetry of such metrics and
  prove that the singularities of the moduli space of left-invariant
  metrics, up to isometric automorphism, is contained in the subspace
  of classes of metrics with maximal index of symmetry.
\end{abstract}

\section{Introduction}

Lie groups endowed with a left invariant metric are objects of great
importance among Riemannian homogeneous spaces. If $G$ is a Lie group
and $g$ is a left invariant metric on $G$, then the geometry of
$(G, g)$ is locally determined, in a simple algebraic way, by choosing
an inner product on the Lie algebra $\mathfrak g$ of $G$. An
interesting problem in Riemannian geometry, which is also important
from the point of view of theoretical physics, is the explicit
computation of the full isometry group of a homogeneous space. If the
Lie group $G$ is nice enough, then the structure of the isometry group
$\I(G, g)$ of $g$ can be recovered, at least locally, from the
underlying algebraic structure. For example, if $G$ is a compact Lie
group, it is proved in \cite{ochiai-takahashi-1976} that the connected
component $\I_0(G, g)$ of the isometry group is a subgroup of
$L(G) \cdot R(G) \subset \Diff(G)$, where $L(G)$ (resp.\ $R(G)$) is
the subgroup of left (resp.\ right) translations on $G$. In a
different context, it was proved by Wolf in \cite{wolf-1963} (see also
\cite{Wilson_1982}) that if $G$ is nilpotent, then
$\I(G, g) \simeq L(G) \rtimes (\Aut(\mathfrak g) \cap \OO(\mathfrak g,
g))$ is the semi-direct product of $L(G)$ and the isometric
automorphisms of $\mathfrak g$, which are identified, via the isotropy
representation, with the isotropy group of $\I(G, g)$ of the identity
element $e$. Notice that in the general case, if $K$ is the full
isotropy group of $e$, then $\I(G, g) = L(G) \cdot K$, but the Lie
group structures of $G$ and $K$ do not determine the structure of
$\I(G, g)$. In fact, it is known (see for instance \cite{shin-1997})
that $L(G)$ is a normal subgroup of $\I(G, g)$ if and only if the
connected component of $K$ is contained in $\Aut(\mathfrak g)$. Thus,
in general, the full isometry group is not a semidirect product of
$L(G)$ and $K$, nor is contained in $L(G) \cdot R(G)$.

This paper is devoted to the case where $G$ has dimension $3$. Recall
that the unimodular case was already treated in
\cite{ha-lee-2012}. Such groups are well behaved in the sense of the
above paragraph. There, the authors take advantage of this in order to
obtain the full isotropy subgroup as the isometric automorphisms of
the Lie algebra. The classification of the isometry groups, in the
unimodular case, relies on the previous classification of Ha-Lee
\cite{ha-lee-2009} of the moduli spaces of left-invariant metrics, up
to isometric automorphism, of $3$-dimensional Lie groups. The explicit
knowledge of these moduli spaces is a starting point for approaching
the non-unimodular case, but the methods used in the unimodular case
are of no use. Some works related to this topic can be mentioned. For
instance, in the article by Lauret \cite{lauret-2003} (see also
\cite{kodama-takahara-tamaru-2011}) there is a classification of Lie
groups with a unique left invariant metric, up to isometric
automorphism and scaling. In $3$ dimensions, a solvable non-unimodular
Lie algebra appears in Lauret's classification, the so-called Lie
algebra of the hyperbolic space, where all the left invariant metrics
are of constant negative curvature. In
\cite{pefoukeu-nimpa-2017-locally}, the authors determine the left
invariant metrics on $3$-dimensional Lie groups which are locally
symmeric. This is done by solving a polynomial system on the structure
coefficients of the underlying Lie algebras, associated to the
parallel curvarute condition. Hence, one can easily determine the full
isometry groups from the eigenvalues of the Ricci tensor. We can also
mention the article of Gordon and Wilson \cite{gordon-wilson-1988} on
transitive isometry subgroups of Riemannian solvmanifolds, which are
in, what they call, standard position inside the full isometry
group. Their results are mainly useful in the case where $G$ is
unimodular solvable.

In this article we compute the full isometry group of any left
invariant metric on a non-unimodular $3$-dimensional Lie group. The
main ideas involved can be summarized as follows. In order to compute
the full isotropy Lie algebra, we use a theorem by Singer, which
describes such Lie algebra as the skew-symmetric endomorphisms of the
tangent space preserving all the derivatives of the curvature tensor,
up to a certain order (see Theorem
\ref{sec:singers-theorem-lie}). Notice that this theorem is stated in
the article \cite{singer-1960} with no proof. A proof can be found in
the later paper \cite{nicolodi-tricerri-1990} by Nicolodi and
Tricerri. In the case where the isotropy group has positive dimension,
we determine the group structure of the full isometry group thinking
of the Killing fields of $G$ as parallel sections of the associated
vector bundle $TG \oplus \so(TG)$, as used in
\cite{console-olmos-2008}. We want to remark two interesting families
of examples that follow from our classification. In the first place,
we explicitly find a curve of simply connected Lie groups with a
left-invariant metric which are isometric but not isomorphic (see
Remark~\ref{sec:case-g_c-1-1}). Secondly, we find a curve of left
invariant metrics on the same Lie group whose isometry groups are not
isomorphic to each other (see Remark~\ref{sec:case-g_0-1}). This
example is somehow surprising since such behaviour occurs only for the
Lie algebra given by the structure coefficients
\begin{align*}
  [e_0, e_1] = 0, && [e_0, e_2] = -e_1, && [e_1, e_2] = -2e_1. 
\end{align*}
In the generic case, the full isometry group consists only of left
translations, or it is the isometry group of a non-compact
symmetric space.

As an application of our results, we compute the index of symmetry of
all the metrics under study. Recall that the index of symmetry is a
geometric invariant, introduced in \cite{olmos-reggiani-tamaru-2014},
which measures how far is a homogeneous metric from being
symmetric. The index of symmetry was successfully computed for several
distinguished families of homogeneous spaces, such as compact naturally
reductive spaces \cite{olmos-reggiani-tamaru-2014} and naturally
reductive nilpotent Lie groups
\cite{reggiani18_distr_symmet_natur_reduc_nilpot_lie_group}, flag
manifolds \cite{podesta-2015} and $3$-dimensional unimodular Lie
groups \cite{Reggiani_2018}. It follows from our results that every
$3$-dimensional Lie group admits a left invariant metric with
non-trivial index of symmetry. We have noticed that in a recent
preprint \cite{may-2021-index}, R. May also compute, independently and
with other techniques, the index of symmtry for these metrics.

Finally, we obtain a result relating the topology of the moduli space
of left invariant metrics with the index of symmetry. More precisely,
we prove that the singularities of such moduli space are contained in
the subsets of (equivalence classes of) metrics with maximal index of
symmetry (cfr.\ Theorem \ref{sec:geom-mean-subs}). We believe these
topological considerations could be extended to other families of
homogeneous spaces in the future.

\section{Preliminaries}
\label{sec:preliminaries}

\subsection{$3$-dimensional, non-unimodular, metric Lie algebras}

The classification of $3$-di\-men\-sion\-al real Lie algebras, up to
isomorphism, has been known for many years (see for instance
\cite{milnor-1976}). There are six isomorphism classes
of unimodular Lie algebras and uncountably many isomorphism classes of
non-unimodular Lie algebras, which can be parametrized in the
following way. Let $\mathfrak g$ be a non-unimodular real Lie algebra
of dimension $3$. Then there exists a basis $e_0, e_1, e_2$ of
$\mathfrak g$ such that
\begin{align}
  \label{eq:2}
  [e_0, e_1]  = 0, && [e_2, e_0] & = \alpha e_0 + \beta e_1,
  && [e_2, e_1] = \gamma e_0 + \delta e_1,
\end{align}
where the matrix $A =
\begin{pmatrix}
  \alpha & \beta \\
  \gamma & \delta
\end{pmatrix}
$ has $\tr A = 2$. Recall that $A$ is the transpose to the matrix of
$\ad_{e_2}$ (when restricted to the subalgebra spanned by $e_0$ and
$e_1$). Moreover, except for the Lie algebra where $A$ is the identity
matrix, the number $c = \det A$ is a complete isomorphism
invariant. In the rest of the article we will denote by
\begin{itemize}
\item $\mathfrak g_I$, the Lie algebra with $A = I$ and
\item $\mathfrak g_c$, the Lie algebra with $A = 
  \begin{pmatrix}
    0 & 1 \\
    -c & 2
  \end{pmatrix}
  $.
\end{itemize}

In the article \cite{ha-lee-2009}, Ha-Lee classifies all the inner
products on $3$-dimensional Lie algebras up to isometric automorphism
and hence, they determine the moduli space of left invariant metrics
for $3$-dimensional Lie groups. We summarize their results, for the
cases of our interest, in Table \ref{tab:left-invariant-metrics},
where such inner products are presented by means of the corresponding
matrix in the basis $e_0, e_1, e_2$. Notice that the expression of the
inner products when $0 < c < 1$ is rather complicated as it involves
the action of the matrix
\begin{equation}
  \label{eq:1}
    P =
    \begin{pmatrix}
      -\frac{1 + \sqrt{1 - c}}{2 c \sqrt{1 - c}} & -\frac{1}{2 \sqrt{1
          - c}} & 0 \\[.4pc]
      \frac{1 - \sqrt{1 - c}}{2 c \sqrt{1 - c}} & \frac{1}{2 \sqrt{1 -
          c}} & 0 \\[.4pc] 
      0 & 0 & 1
    \end{pmatrix} \in \GL_3(\mathbb R)
  \end{equation}
on a non-diagonal, symmetric, positive definite matrix.

We shall denote by $G_I$ and $G_c$ the simply connected Lie groups
with Lie algebra $\mathfrak g_I$ and $\mathfrak g_c$,
respectively. Since these Lie algebra are solvable, $G_I$ and $G_c$
are diffeomorphic to $\mathbb R^3$. The corresponding Lie group
structures can be described in the following way (as presented in
\cite{ha-lee-2009}). In all cases the Lie group will be a semi-direct
product $\mathbb R^2 \rtimes_\varphi \mathbb R$, where the
representation $\varphi: \mathbb R \to \GL_2(\mathbb R)$ is given by
\begin{equation}\label{eq:16}
  \varphi(t) =
  \begin{cases}
    \left(
      \begin{smallmatrix}
        e^t & 0 \\
        0 & e^t
      \end{smallmatrix}
    \right)
    & \text{ for } G_I \\
    e^t\left(
      \begin{smallmatrix}
        \cosh(t\sqrt{1 - c})
      \end{smallmatrix}
      \left(
        \begin{smallmatrix}
          1 & 0\\
          0 & 1
        \end{smallmatrix}
      \right) + \frac{\sinh(t\sqrt{1 - c})}{\sqrt{1 - c}}
      \left(
        \begin{smallmatrix}
          -1 & - c \\
          1 & 1 \\
        \end{smallmatrix}
      \right)
    \right)
    & \text{ for } G_c
  \end{cases}
\end{equation}
and hence the product on $G_c$ is given by
\begin{equation}
  \label{eq:4}
  (v, t)(w, s) = (v + \varphi(t)w, t + s).
\end{equation}

Note that the formula in (\ref{eq:16}) makes sense even if $c > 1$.

The left invariant metrics on $G_\bullet$, where
$\bullet \in \{I, c\}$, associated with the inner products on
$\mathfrak g_\bullet$ listed in Table \ref{tab:left-invariant-metrics}
will be denoted by the same symbol
$g \in \{g_\nu, g_{\mu, \nu}, g'_{\lambda, \nu}\}$. This presents
certain ambiguity as $g_\nu$ and $g_{\mu, \nu}$ have different meaning
for different Lie algebras, but such ambiguity disappears when one
specifies the value of $\bullet$.

\begin{table}[ht]
  \caption{Left invariant metrics up to isometric
    automorphism on $3$-dimensional, non-unimodular Lie groups.}
  \centering
  {\tabulinesep=1.2mm
    \begin{tabu}{|c|c|c|c|}
      \hline
      Lie algebra & \multicolumn{2}{c|}{Left invariant metrics} & Parameters
      \\ \hline \hline
      $\mathfrak g_I$ & \multicolumn{2}{c|}{$g_\nu = \left(
          \begin{smallmatrix}
            1 & 0 & 0 \\
            0 & 1 & 0 \\
            0 & 0 & \nu
          \end{smallmatrix}
        \right)$} & $0<\nu$ \\ \hline
      \multirow{7}{*}{$\mathfrak g_c$} & $c<0$ & $g_{\mu,\nu}=\left(
        \begin{smallmatrix}
          1&0&0\\
          0&\mu&0\\
          0&0&\nu
        \end{smallmatrix}
      \right)$ & $0<\mu\le|c|$ and $0<\nu$ \\ \cline{2-4} 
      & \multirow{2}{*}{$c=0$} & $g_{\mu,\nu}=\left(
        \begin{smallmatrix}
          1&0&0\\
          0&\mu&0\\
          0&0&\nu
        \end{smallmatrix}
      \right)$ & $0<\mu,\nu$ \\ \cline{3-4} 
      & & $g_\nu=\left(
        \begin{smallmatrix}
          1&\frac12&0\\
          \frac12&1&0\\
          0&0&\nu
        \end{smallmatrix}
      \right)$ & $0<\nu$ \\ \cline{2-4} 
      & $0<c<1$ & $g_{\mu,\nu}=P^T\left(
        \begin{smallmatrix}
          1&\mu&0\\
          \mu&1&0\\
          0&0&\nu
        \end{smallmatrix}
      \right)P$ & {$P$ as in (\ref{eq:1}), $0\le\mu<1$ and $0<\nu$} \\
      \cline{2-4} & \multirow{2}{*}{$c=1$} & $g_{\mu,\nu}=\left(
        \begin{smallmatrix}
          1&0&0\\
          0&\mu&0\\
          0&0&\nu
        \end{smallmatrix}
      \right)$ & $0<\mu\le1$ and $0<\nu$ \\ \cline{3-4}
      & & $g'_{\lambda,\nu}=\left(
        \begin{smallmatrix}
          1&\lambda&0\\
          \lambda&1&0\\
          0&0&\nu
        \end{smallmatrix}
      \right)$ & $0<\lambda<1$ and $0<\nu$ \\ \cline{2-4} 
      & $1<c$ & $g_{\mu,\nu}=\left(
        \begin{smallmatrix}
          1&1&0\\
          1&\mu&0\\
          0&0&\nu
        \end{smallmatrix}
      \right)$ & $1<\mu\le c$ and $0<\nu$ \\ \hline
\end{tabu}
}
\label{tab:left-invariant-metrics}
\end{table}

\subsection{A theorem by Singer on the Lie algebra of the isometry group}

Let $\mathbb V$ be an Euclidean space and let $T$ be an algebraic
tensor of type $(s, t)$ on $\mathbb V$. The natural action of the
orthogonal Lie algebra $\so(\mathbb V)$ on $T$ is defined by
\begin{align*}
  (A \cdot T)(\omega_1, \ldots, \omega_s, v_1, \ldots, v_t)
  & = - \sum_{i = 1}^s T(\omega_1, \cdots, \omega_{i-1}, \omega_i
    \circ A, \omega_{i + 1}, \ldots, \omega_s, v_1, \ldots, v_t) \\
  & \qquad + \sum_{j = 1}^t T(\omega_1, \ldots, \omega_s, v_1, \ldots,
    v_{j - 1}, Av_j, v_{j + 1}, \ldots, v_t),
\end{align*}
where $A \in \so(\mathbb V)$, $\omega_1 \ldots, \omega_s \in \mathbb
V^*$ and $v_1, \ldots v_t \in \mathbb V$. In the work
\cite{singer-1960}, Singer discusses under what conditions a family of
algebraic tensors $R^s$ ($s \ge 0$) of type $(1, s + 3)$ on $\mathbb V = T_pM$,
the tangent space of a differentiable manifold $M$, coincides with the
covariant derivatives of order $s$ at $p$ of the curvature tensor of some
homogeneous Riemannian metric on $M$. In particular, if $s$ ranges
between $0$ and certain integer $k$, one can theoretically recover the
Lie algebra of the isometry group of $M$ from the action of
$\so(T_pM)$ on the algebraic tensors $R_p, (\nabla R)_p, \ldots,
(\nabla^{k + 1} R)_p$. We find it useful to rephrase such theorem in
order to give a procedure to find the Lie algebra of the isotropy
subgroup of the full isometry group of $M$.

\begin{theorem}
  [{Singer \cite{singer-1960}}] \label{sec:singers-theorem-lie} Let
  $M$ be a homogeneous Riemannian space of dimension $n$. Let
  $p \in M$ and $A \in \so(T_pM)$. Then, there exists a Killing vector
  field $X$ on $M$ such that $X_p = 0$ and $(\nabla X)_p = A$ if and
  only if $A \cdot (\nabla^s R)_p = 0$ for all
  $0 \le s < \frac{n (n - 1)}{2}$, where $R$ denotes the Riemannian
  curvature tensor of $M$.
\end{theorem}

One can improve the statement of the above theorem. In fact, it is
enough to assume that $s$ takes values between $0$ and the so-called
Singer invariant of $M$ (\cite{singer-1960}, see also
\cite{nicolodi-tricerri-1990,console-olmos-2008}).

\subsection{The index of symmetry of a homogeneous
  manifold}\label{sec:index-symm-homog}

Let $(M, g)$ be a Riemannian homogenous space and let $\I(M, g)$ be
its full isometry Lie group. The Lie algebra of $\I(M, g)$ is
canonically identified with the Lie algebra of Killing vector fields
$\mathcal K(M, g)$. The \emph{distribution of symmetry}
$\mathfrak s_g$ of $(M, g)$ is defined at the point $p \in M$ by
\begin{equation*}
  \mathfrak s_g = \{v \in T_pM : v = X_p \text{ for some } X \in
  \mathcal K(M, g) \text{ with }(\nabla X)_p = 0\}.
\end{equation*}

One has that $\mathfrak s_g$ is an autoparallel distribution of $M$,
that is, $\mathfrak s_g$ is integrable with totally geodesics
leaves. Moreover, the distribution of symmetry is invariant under
$\I(M, g)$ and its integrable manifolds are extrinsically isometric to
a globally symmetric space, which is called the \emph{leaf of
  symmetry} of $(M, g)$ (see for instance
\cite{olmos-reggiani-tamaru-2014,berndt-olmos-reggiani-2017}). The
rank of $\mathfrak s_g$ is usually called the \emph{index of symmetry}
of $(M, g)$ and it is denoted by $i_{\mathfrak s}(M, g)$. The index of
symmetry is a geometric invariant which measures how far is $(M, g)$
from being a symmetric space, in the sense that $i_{\mathfrak s}(M, g)
= \dim M$ if and only if $(M, g)$ is a symmetric space. It is known
that the distribution of symmetry can never have corank equal to $1$
(see \cite{Reggiani_2018}). In particular, if $M$ has dimension $3$
and $(M, g)$ is not a symmetric space, then the index of symmetry is
equal to $0$ or $1$.

\section{The full isometry groups}
\label{sec:full-isometry-groups}

In this section we compute the full isometry groups of
$(G_\bullet, g)$ where $\bullet \in \{I\} \cup \mathbb R$ and $g$ is a
left invariant metric on $G_\bullet$. Recall that we will be using the
notations given in Section~\ref{sec:preliminaries}. We analyze several
cases according to the structure of the Lie algebra of $G_\bullet$ and
the moduli space of left invariant metrics. Before getting into the
case-by-case analysis, we outline the general idea for dealing with
the generic case. Assume that the left invariant metric $g$ on
$G_\bullet$ is not symmetric, and let $G = \I(G_\bullet, g)$. Notice
that, $\dim G \le 4$. In fact, the isotropy group of $G$ at any point
is isomorphic, via the isotropy representation, to a compact subgroup
of $\OO(3)$. Let $e \in G_\bullet$ be the identity element. If
$\varphi \in G$ and $\varphi(e) = x$, then
$\varphi = L_x\circ (L_{x^{-1}} \circ \varphi)$, where $L_x \in G$ is
the left translation by $x$. So, the full isometry group decomposes as
\begin{equation}\label{eq:14}
  G \simeq G_\bullet \cdot G_e,
\end{equation}
where $G_e$ is the isotropy subgroup of $G$ at $e$ and we are
identifying $G_\bullet$ with $L(G_\bullet)$. Notice that, in general,
(\ref{eq:14}) is not a semi-direct product. In fact, according to
\cite[Lemma~1.1]{shin-1997}, $G_\bullet$ is a normal subgroup of the
connected component of $G$ if and only of the connected component of
$G_e$ is contained in $\Aut(\mathfrak g_\bullet)$, under the usual
identifications. However, one can recover the algebraic structure of
$G$ from the group structure of $G_\bullet$ and $G_e$ and the geometry
of the metric $g$. In fact, identify the Lie algebra of $G$ with the
Lie algebra of Killing fields $\mathcal K(G_\bullet, g)$. We have from
(\ref{eq:14}) that every element of $\mathcal K(G_\bullet, g)$ can be
written as $X + Y$ where $X$ is a right invariant vector field on
$G_\bullet$ and $Y$ is a Killing field such that $Y_e = 0$. On the
other hand, if $X, X'$ are two Killing vector fields with initial
conditions $X_e = v$, $X'_e = v'$ and $(\nabla X)_e = B$,
$(\nabla X')_e = B'$, then the initial comditions of the Lie bracket
$[X, X']$ are
\begin{align}\label{eq:15}
  [X, X']_e = B'v - Bv', && (\nabla [X, X'])_e = R_{v, v'} - [B, B'].
\end{align}
In fact, the first condition follows from the fact that the
Levi-Civita connection is torsion free, and the second one can be
derived from the so called Killing affine equation (see
\cite{console-olmos-2008} or \cite{Reggiani_2018} for more details).
It follows that in order to compute the full isometry group one only
needs to compute $G_e$. We identify $G_e$, via the isotropy
representation, with a subgroup of
$\OO(\mathfrak g_\bullet, g) \simeq \OO(3)$. At the Lie algebra level,
one can recover the connected component of $G_e$ from its Lie algebra
$\mathfrak g_e \subset \so(\mathfrak g_\bullet, g) \simeq \so(3)$. In
order to simplify some long calculations, it is useful to note
that an element of $G_e$ must preserve the Ricci tensor. In most
cases, the Ricci tensor of $G_\bullet$ is semi-definite and therefore
if an element $A \in \so(\mathfrak g_\bullet, g)$ is induced by a
one-parameter subgroup of $G_e$, then $A$ belongs to a Lie subalgebra
$\mathfrak h \subset \mathfrak{gl}(\mathfrak g_\bullet)$ which is
isomorphic to $\so(2,1)$. This reduces the complexity of the problem,
as we now only need to look for
\begin{equation*}
  A \in \so(\mathfrak g_\bullet, g) \cap \mathfrak h,
\end{equation*}
which in the generic case, has dimension at most $1$. In the case
that this dimension is positive, we then can check if $A$ is in fact
induced by isometries by using Theorem \ref{sec:singers-theorem-lie}.

We now proceed with the study of all the possible cases. We deal with
the isolated cases first.

\subsection{The case of $G_I$}

From (\ref{eq:2}), the bracket relations for $\mathfrak g_I$ are
\begin{align*}
  [e_0, e_1] = 0, && [e_2, e_0] = e_0, && [e_2, e_1] = e_1,
\end{align*}
and so, $\mathfrak g_I$ is isomorphic to the so called \emph{Lie
  algebra of the hyperbolic space} $H^3$. That is, the Lie algebra of
the solvable, transitive group of isometries of $H^3$. Notice that
this Lie group admits a left-invariant metric which makes it isometric
to $H^3$. Moreover, according to \cite{lauret-2003} (see also
\cite{kodama-takahara-tamaru-2011}) this is the only possible left
invariant metric for $G_I$ up to automorphism and scaling. This proves
the following result.

\begin{theorem}
  \label{sec:case-g_i}
  Let $g$ be a left invariant metric on $G_I$, then
  \begin{equation*}
    \I(G_I, g) \simeq \SO(3, 1).
\end{equation*}
\end{theorem}

\begin{remark}
  It follows from routine calculations that if we take $g = g_\nu$ in
  the above theorem, then $(G_I, g_\nu)$ is isometric to the real
  hyperbolic space of curvature $-\frac 1 \nu$ in three dimensions. 
\end{remark}

\subsection{The case of $G_0$}

According to (\ref{eq:4}), the Lie group structure of $G_0 \simeq
\mathbb R^2 \rtimes \mathbb R$ is given by
\begin{equation*}
  (x_0, x_1, x_2)(y_0, y_1, y_2) = \left(x_0 + y_0, x_1 + \tfrac{e^{2 \, x_{2}}}{2} \,
    {\left(y_0 + 2  y_1\right)}  -
    \tfrac{y_0}{2}  , x_2 + y_2\right) 
\end{equation*}
and the canonical frame of left invariant vector fields is given by
\begin{align*}
  e_{0} = \frac{\partial}{\partial x_{0} } + \frac{e^{x_{2}} - 1}{2} 
  \frac{\partial}{\partial x_{1} },
  &&
     e_{1} = e^{2 \, x_{2}} \frac{\partial}{\partial x_{1} },
  &&
     e_{2} = \frac{\partial}{\partial x_{2} }.
\end{align*}

We also have the explicit expression for the corresponding right
invariant vector fields:
\begin{align*}
  r_0 = \frac{\partial}{\partial x_0},
  &&
     r_1 = \frac{\partial}{\partial x_1},
  &&
     r_{2} = (x_0 + 2 x_1) \frac{\partial}{\partial x_1} +
     \frac{\partial}{\partial x_2}.
\end{align*}

\begin{theorem}
  \label{sec:case-g_0}
  Keeping the notation of Table \ref{tab:left-invariant-metrics}, we
  have that
  \begin{enumerate}
  \item\label{item:1} $\I(G_0, g_{\mu, \nu}) \simeq G_0 \cdot \SO(2)$
    where its Lie algebra has a basis $r_0, r_1, r_2, A$ such that
    \begin{align*}
      [r_0, r_1] &= 0, & [r_0, r_2] &= r_1, & [r_1, r_2] & = 2 r_1, \\
      [r_0, A] &= r_2, & [r_1, A] &= 2 r_2, & [r_2, A] & = -\nu r_0 - \frac{2 \nu}{\mu} r_1 + 2 A.
    \end{align*}
  \item\label{item:2} $\I(G_0, g_\nu) \simeq \E(1) \times \SO(2, 1)$,
    where $\E(1)$ is the Euclidean group in one dimension.
  \end{enumerate}
\end{theorem}

\begin{proof}
  Let us consider first the case of $g = g_{\mu, \nu}$. It is not
  difficult to see that if $X$ is a left invariant vector field on
  $G_0$, that is also a Killing vector field, then $X$ must be a
  scalar multiple of $e_0 - \frac12 e_1 = r_0 - \frac12 r_1$, which is
  also right invariant. So, there are no new Killing vector fields
  arising as the difference of a left and a right invariant vector
  fields taking the same value at the identity.

  In order to compute the isometry group we first compute the
  orthogonal Lie algebra
  \begin{equation}\label{eq:6}
    \so(\mathfrak g_0, g_{\mu, \nu}) =
    \left\{
      \begin{pmatrix}
        0 & -a_{10} \mu & -a_{20} \nu \\[.4pc]
        a_{10} & 0 & -\frac{a_{21} \nu}{\mu} \\[.4pc]
        a_{20} & a_{21} & 0
      \end{pmatrix}: a_{10}, a_{20}, a_{21} \in \mathbb R
    \right\}.
  \end{equation}
  Now we compute the Ricci tensor of the metric, which in the left
  invariant frame $e_0, e_1, e_2$ takes the form
  \begin{equation*}
    \Ric =
    \begin{pmatrix}
      -\frac{\mu}{2 \, \nu} & -\frac{2 \, \mu}{\nu} & 0 \\[.4pc]
      -\frac{2 \, \mu}{\nu} & \frac{\mu(\mu - 8)}{2 \, \nu} & 0 \\[.4pc]
      0 & 0 & -\frac{\mu - 8}{2} 
    \end{pmatrix}
  \end{equation*}
  and it induces a non-degenerate, semi-definite, symmetric bilinear
  form on $\mathfrak g_0$. The connected component of the isotropy
  group is then identified with the inteserction
  \begin{equation}
    \so(\mathfrak g_0, g_{\mu, \nu}) \cap \so(\mathfrak g_0, \Ric) =
    \mathbb R A,
  \end{equation}
  where $A$ is the matrix
  \begin{equation}\label{eq:5}
    A =
    \begin{pmatrix}
      0 & 0 & - \nu \\
      0 & 0 & -\frac{2 \nu}{\mu} \\
      1 & 2 & 0
    \end{pmatrix}.
  \end{equation}
  Let us denote by $R$ the Riemannian curvature tensor of
  $g_{\mu, \nu}$. It is relatively easy to see that $A \cdot R_e = 0$
  and after some heavy calculations\footnote{Calculations for this
    paper where verified using the extension SageManifolds of the
    computer algebra system SageMath. The corresponding
    jupyter-notebooks can be downloaded from
    \href{https://www.fceia.unr.edu.ar/~reggiani/isom-dim-3/}{\url{https://www.fceia.unr.edu.ar/~reggiani/isom-dim-3/}}
  } we can check that $A \cdot (\nabla R)_e = 0$ and
  $A \cdot (\nabla^2 R)_e = 0$. Hence, by Theorem
  \ref{sec:singers-theorem-lie}, $A = (\nabla Z)_e$ for some Killing
  field such that $Z_e = 0$. This implies that the connected component
  of the isotropy group is a subgroup of
  $\SO(\mathfrak g_0, g_{\mu, \nu})$ isomorphic to $\SO(2)$. Since
  $\dim G_0 = 3$ and the metric is not symmetric, we get that the
  isotropy group is connected and so
  $\I(G_0, g_{\mu, \nu}) \simeq G_0 \cdot \SO(2)$. Otherwise, there
  would be non trivial isometries in the two connected componnets of
  $\OO(\mathfrak g_0, g_{\mu, \nu})$, which would imply that the
  geodesic symmetry is an isometry. In order to determine the Lie
  group structure of the isometry group, we identify
  $\mathcal K(G_0, g_{\mu, \nu}) \simeq \mathfrak g_0^r \oplus \mathbb
  R Z$ (direct sum of vector spaces), where $\mathfrak g_0^r$ is the
  Lie algebra of right invariant vector fields on $G_0$. Let us
  $r_0, r_1, r_2$ be the basis of right invariant vector fields
  definied above the theorem. In order to compute the brackets
  $[r_i, Z]$, we use the identities (\ref{eq:15}) and the fact that
  \begin{align*}
    (\nabla r_0)_e =
    \begin{pmatrix}
      0 & 0 & 0 \\
      0 & 0 & -\frac{1}{2} \\
      0 & \frac{\mu}{2 \, \nu} & 0
    \end{pmatrix},
        &&
           (\nabla r_1)_e =
           \begin{pmatrix}
             0 & 0 & -\frac{1}{2} \, \mu \\
             0 & 0 & -2 \\
             \frac{\mu}{2 \, \nu} & \frac{2 \, \mu}{\nu} & 0
           \end{pmatrix},
        &&
           (\nabla r_2)_e =
           \begin{pmatrix}
             0 & -\frac{1}{2} \, \mu & 0 \\
             \frac{1}{2} & 0 & 0 \\
             0 & 0 & 0
           \end{pmatrix}.
  \end{align*}
  
  Now we direct our attention to the metric $g = g_\nu$. If we define
  a new left invariant vector frame by
  \begin{align*}
    \hat e_0 = 2 e_0 - e_1,
    && \hat e_1 = \frac 12 \, e_2,
    && \hat e_2 = e_1, 
  \end{align*}
  then $\mathfrak z(\mathfrak g_0) = \mathbb R \hat e_0$ and $[\hat
  e_1, \hat e_2] = \hat e_2$. So $\mathfrak g_0 \simeq \mathbb R
  \oplus \mathfrak g_{\mathbb RH^2}$ (direct sum of Lie algebras),
  where $\mathfrak g_{\mathbb RH^2}$ is the Lie algebra of the real
  hyperbolic space $\mathbb RH^2$. Moreover, the left invariant metric
  $g_\nu$ takes now the form
  \begin{equation*}
    g_\nu = 3 \, \hat e^{0}\otimes \hat e^{0} + \frac{1}{4} \, \nu
    \hat e^{1}\otimes \hat e^{1} + \hat e^{2}\otimes \hat e^{2}.
  \end{equation*}
  Thus, $g_\nu$ splits off its center and $(G_0, g_\nu)$ is isometric
  to (a rescaling of) the Riemannian product
  $\mathbb R \times \mathbb RH^2$. In fact, as we noticed in the proof
  of Theorem \ref{sec:case-g_i}, the restriction of $g_\nu$ to
  $\mathfrak g_{\mathbb RH^2}$ is unique up to automorphism and
  scaling. Therefore, the isometry group of $(G_0, g_\nu)$ is
  isomorphic to $\E(1) \times \SO(2, 1)$.
\end{proof}

\begin{remark}
  \label{sec:case-g_0-1}
  If two left invariant metrics on $G_0$ are not equivalent up to
  isometric automorphism or scaling, then the corresponding isometry
  groups are not isomorphic. In fact, the case when one of the metrics
  is equivalent of some $g_\nu$ is trivial. So we only need to
  consider metrics of the form $g_{\mu, \nu}$, and we can assume
  further that $\nu = 1$. In these conditions, the Killing form
  distinguish the Lie algebras given in item \ref{item:1} of Theorem
  \ref{sec:case-g_0}. In fact, it is easy to see that if $\nu = 1$,
  the eigenvalues of the Killing form of the isometry Lie algebra of
  $g_{\mu, 1}$ are
  \begin{equation*}
    -\frac{\mu + 4 \pm \sqrt{81 \mu^{2} + 8 \mu + 16}}{\mu}, 0, 8.
  \end{equation*}
\end{remark}

\subsection{The case of $G_1$}
Taking $c = 1$ in (\ref{eq:4}) we can compute explicit expressions for
the left and right invariant vector frames, which will be used in
further calculations. The left invariant vector frame $e_0, e_1, e_2$
is given by
\begin{align*}
  e_0 = {\left(1 - x_{2}\right)} e^{x_{2}} \frac{\partial}{\partial
    x_{0} } + x_{2} e^{x_{2}} \frac{\partial}{\partial x_{1} },
  && e_1 = -x_{2} e^{x_{2}} \frac{\partial}{\partial x_{0} } +
     {\left(1 + x_{2}\right)} e^{x_{2}} \frac{\partial}{\partial x_{1}
     },
  && e_2 = \frac{\partial}{\partial x_{2} },
\end{align*}
and the corresponding right invariant vector frame is
\begin{align*}
  r_0 = \frac{\partial}{\partial x_{0} },
  && r_1 = \frac{\partial}{\partial x_{1} },
  && r_2 = -x_{1} \frac{\partial}{\partial x_{0} } + \left( x_{0} + 2
     x_{1} \right) \frac{\partial}{\partial x_{1} }
     +\frac{\partial}{\partial x_{2} }.
\end{align*}

\begin{theorem} Let $g$ be a left invariant metric on $G_1$, then
  \label{sec:case-g_1}
  \begin{equation*}
    \I(G_1, g) \simeq G_1.
  \end{equation*}
\end{theorem}

\begin{proof}
  Assume first that $g = g_{\mu, \nu}$ in
  Table~\ref{tab:left-invariant-metrics}. Notice that in the frame
  $e_0, e_1, e_2$  the orthogonal Lie algebra has the same matrix
  representation as in (\ref{eq:6}), but with the constrain $0 < \mu
  \le 1$. The Ricci tensor of $g_{\mu, \nu}$ is represented in our
  frame by the matrix
  \begin{equation*}
    \Ric =
    \begin{pmatrix}
      -\frac{\mu^{2} - 1}{2 \, \mu \nu} & -\frac{2 \, \mu}{\nu} & 0 \\
      -\frac{2 \, \mu}{\nu} & \frac{\mu^{2} - 8 \, \mu - 1}{2 \, \nu} & 0 \\
      0 & 0 & -\frac{\mu^{2} + 6 \, \mu + 1}{2 \, \mu}
    \end{pmatrix}.
  \end{equation*}
  After standard computations we get that
  $A \in \so(\mathfrak g_1, g_{\mu, \nu})$ is also an
  element of $\so(\mathfrak g_1, \Ric)$ if and only if $a_{10} = 0$ and
  \begin{equation*}
    a_{20}(1 + 3\mu) - 2 \, a_{21} = 2 \,
    a_{20} \mu + a_{21}(1 - \mu)  = 0
  \end{equation*}
  which trivially implies $a_{20} = a_{21} = 1$ (recall that we are
  assuming that $A$ has the form given in (\ref{eq:6})). Since the
  metric is not symmetric, we get that the isotropy group of
  $\I(G_1, g_{\mu, \nu})$ is trivial.

  Finally, for the case of $g = g'_{\lambda, \nu}$ we have
  \begin{equation*}
    \so(\mathfrak g_1, g'_{\lambda, \nu}) =
    \left\{
      \begin{pmatrix}
        -a_{11} & -\frac{a_{11}}{{\lambda}} & -\frac{{\left(a_{21}
              {\lambda} - a_{20}\right)} \nu}{{\lambda}^{2} - 1} \\
        \frac{a_{11}}{{\lambda}} & a_{11} & -\frac{{\left(a_{20}
              {\lambda} - a_{21}\right)} \nu}{{\lambda}^{2} - 1} \\ 
        a_{20} & a_{21} & 0
      \end{pmatrix}: a_{11}, a_{20}, a_{21} \in \mathbb R
    \right\}
  \end{equation*}
  and
  \begin{equation}\label{eq:7}
    \Ric =
    \begin{pmatrix}
      -\frac{4 \, {\lambda}}{{\left({\lambda} + 1\right)} \nu} &
      -\frac{2 \, {\left({\lambda}^{2} + 1\right)}}{{\left({\lambda} +
            1\right)} \nu} & 0 \\ 
      -\frac{2 \, {\left({\lambda}^{2} + 1\right)}}{{\left({\lambda} +
            1\right)} \nu} & \frac{4 \, {\left({\lambda}^{2} -
            {\lambda} - 1\right)}}{{\left({\lambda} + 1\right)} \nu} &
      0 \\  
      0 & 0 & -\frac{4}{{\lambda} + 1}
    \end{pmatrix}.
  \end{equation}
  Reasoning as in the above paragraph we prove that
  $\so(\mathfrak g_1, g'_{\lambda, \nu}) \cap \so(\mathfrak g_1, \Ric)
  = 0$. Notice that we are making here some abuse of notation, since
  $\Ric$ in (\ref{eq:7}) is degenerate when $\lambda = \sqrt 5 -
  2$. Using again that the metric is not symmetric, we conclude that
  the full isometry group is isomorphic to $G_1$.
\end{proof}

\subsection{The case of $G_c$, $c < 0$}

In order to simplify the notation and make some calculations easier, we
write 
\begin{equation*}
  c = 1 - c_1^2, \qquad \text{ for } c_1 > 1.
\end{equation*}

This trick is very useful when checking our calculations with
SageMath, since this software finds it very hard simplifying certain
expressions where $\sqrt{1 - c}$ appears, which now are just replaced
by $c_1$.  The left invariant frame in which the metrics of Table
\ref{tab:left-invariant-metrics} are represented is given by
\begin{align}
  e_0 &= e^{x_2}\left(\cosh(c_1 x_2) - \frac{\sinh(c_1 x_2)}{c_1} \right)
  \frac{\partial}{\partial x_0} +  \frac{e^{x_2} \sinh(c_1 x_2)}{c_1}
        \frac{\partial}{\partial x_1}, \notag \\
  e_1 & = -\frac{c \, e^{x_2} \sinh(c_1 x_2)}{c_1}
        \frac{\partial}{\partial x_0} + e^{x_2}\left(\cosh(c_1 x_2) +
        \frac{\sinh(c_1 x_2)}{c_1}\right) \frac{\partial}{\partial
        x_1}, \label{eq:8} \\
  e_2 & = \frac{\partial}{\partial x_2}, \notag
\end{align}
and the associated right invariant vector frame is
\begin{align}\label{eq:9}
  r_0 = \frac{\partial}{\partial x_0},
  && r_1  = \frac{\partial}{\partial x_1},
  && r_2  = -c x_1 \frac{\partial}{\partial x_0} + \left(x_0 +
     2x_1\right)\frac{\partial}{\partial x_1} + \frac{\partial}{\partial
     x_2}. 
\end{align}

\begin{theorem}
  \label{sec:case-g_c-c}
  If $c < 0$, then
  \begin{equation*}
    \I(G_c, g) \simeq G_c
  \end{equation*}
\end{theorem}

\begin{proof}
  The proof is similar to the ones of previous cases. We may assume
  that $g = g_{\mu, \nu}$ and hence the orthogonal Lie algebra can be
  presented as in (\ref{eq:6}). For brevity, we present here the Ricci
  tensor in terms of $c$ (instead of $c_1$):
  \begin{equation*}
    \Ric =
    \begin{pmatrix}
      \frac{c^2 - \mu^{2}}{2 \, \mu \nu} & -\frac{2 \,
        \mu}{\nu} & 0 \\ 
      -\frac{2 \, \mu}{\nu} & -\frac{c^2 + 8 \mu (1 -
        \mu)}{2 \, \nu} & 0 \\ 
      0 & 0 & -\frac{(c - \mu)^2 -8 \mu}{2 \, \mu}
    \end{pmatrix}.
  \end{equation*}
  Now it is not hard to see that $\so(\mathfrak g_c, g_{\mu, \nu})
  \cap \so(\mathfrak g_c, \Ric) = 0$. In particular, this shows that
  the metric is not symmetric and hence $\I(G_c, g_{\mu, \nu}) \simeq G_c$.
\end{proof}

\subsection{The case of $G_c$, $0 < c < 1$}
As in the previous case, we write $c = 1 - c_1^2$ and hence the left
and right invariant vector frames have now the exact same form as in
(\ref{eq:8}) and (\ref{eq:9}), respectively, but now $0 < c_1 < 1$.

\begin{theorem}
  \label{sec:case-g_c-0}
  If $0 < c  < 1$ and $g$ is a left invariant metric on $G_c$, then
  \begin{equation*}
    \I(G_c, g) \simeq G_c.
  \end{equation*}
\end{theorem}

\begin{proof}
  We take $g = g_{\mu, \nu}$. Recall that in this case the
  coefficients of the metric with respect to the frame $e_0, e_1, e_2$
  are rather involved. In fact, it needs to be computed as in Table
  \ref{tab:left-invariant-metrics} using the matrix $P$ as
  in~(\ref{eq:1}). After some suitable simplifications we can express,
  in the usual frame, the metric as
  \begin{equation}\label{eq:12}
    g_{\mu, \nu} =
    \begin{pmatrix}
      \frac{c \mu + c - 2}{2 \, {\left(c - 1\right)} c^{2}} & \frac{\mu - 1}{2 \, {\left(c - 1\right)} c} & 0 \\
      \frac{\mu - 1}{2 \, {\left(c - 1\right)} c} & \frac{\mu - 1}{2 \, {\left(c - 1\right)}} & 0 \\
      0 & 0 & \nu
    \end{pmatrix}
  \end{equation}
  and the Ricci tensor as
  \begin{equation*}
    \Ric = 
    \begin{pmatrix}
      \frac{{\left(\mu^{2} + \mu\right)} c^{2} + 2 \, \mu^{2} -
        {\left(\mu^{3} + 2 \, \mu^{2} + 1\right)} c}{(1 - \mu^2) \nu
        {\left(1 - c\right)} c^{2}}
      & \frac{\mu^{2} - c}{{\left(\mu + 1\right)} \nu {\left(1 -
            c\right)} c} & 0 \\  
      \frac{\mu^{2} - c}{{\left(\mu + 1\right)} \nu {\left(1 - c\right)} c}
      & -\frac{\mu^{2} + \mu t - \mu - 1}{{\left(\mu + 1\right)} \nu
        {\left(t - 1\right)}} & 0 \\
      0 & 0 & \frac{2 \, {\left(\mu^{2} + c - 2\right)}}{{\left(1 -
            \mu^2\right)}} 
    \end{pmatrix}.
  \end{equation*}

  The same argument used previously give us that
  $\so(\mathfrak g_c, g_{\mu, \nu}) \cap \so(\mathfrak g_c, \Ric) = 0$
  and since the metric is not symmetric,
  $\I(G_c, g_{\mu, \nu}) \simeq G_c$.
\end{proof}

\subsection{The case of $G_c$, $1 < c$}

We write $c = 1 + c_1^2$ with $c_1 > 0$. We can perform some formal
manipulation in (\ref{eq:8}) in order to compute the left invariant
frame for $\mathfrak g_c$ in an easy way. In fact, rewriting
$c = 1 - (i c_1)^2$, we can present the left invariant vector fields
as in (\ref{eq:8}) replacing $c_1$ by $i c_1$. Now, from
$\cosh(iz) = \cos z$ and $\sinh(iz) = i \sin z$ follows that
\begin{align}
  e_0 &= e^{x_2}\left(\cos(c_1 x_2) - \frac{\sin(c_1 x_2)}{c_1} \right)
  \frac{\partial}{\partial x_0} +  \frac{e^{x_2} \sin(c_1 x_2)}{c_1}
        \frac{\partial}{\partial x_1}, \notag \\
  e_1 & = -\frac{c \, e^{x_2} \sin(c_1 x_2)}{c_1}
        \frac{\partial}{\partial x_0} + e^{x_2}\left(\cos(c_1 x_2) +
        \frac{\sin(c_1 x_2)}{c_1}\right) \frac{\partial}{\partial
        x_1}, \label{eq:8} \\
  e_2 & = \frac{\partial}{\partial x_2}, \notag
\end{align}
 and the associated right invariant vector frame is
\begin{align}\label{eq:9}
  r_0 = \frac{\partial}{\partial x_0},
  && r_1  = \frac{\partial}{\partial x_1},
  && r_2  = -c x_1 \frac{\partial}{\partial x_0} + (x_0 +
     2x_1)\frac{\partial}{\partial x_1} + \frac{\partial}{\partial
     x_2}. 
\end{align}

\begin{theorem}
  \label{sec:case-g_c-1}
  If $1 < c$ and $g_{\mu, c}$ is the left invariant metric on $G_c$
  described in Table \ref{tab:left-invariant-metrics}, then
  \begin{equation*}
    \I(G_c, g_{\mu, \nu} ) \simeq
    \begin{cases}
      G_c & \text{ if } 1 < \mu < c, \\
      \SO(3, 1) & \text{ if } \mu = c.
    \end{cases}
\end{equation*}
\end{theorem}

\begin{proof}
  The orthogonal Lie algebra of $g_{\mu, \nu}$ is given by
  \begin{equation}\label{eq:10}
    \so(\mathfrak g_c, g_{\mu,\nu}) =
    \left\{
      \begin{pmatrix}
        -a_{11} & -a_{11} \mu & -\frac{{\left(a_{20} \mu - a_{21}\right)} \nu}{\mu - 1} \\
        a_{11} & a_{11} & \frac{{\left(a_{20} - a_{21}\right)} \nu}{\mu - 1} \\
        a_{20} & a_{21} & 0
      \end{pmatrix}: a_{11}, a_{20}, a_{21} \in \mathbb R
    \right\}
  \end{equation}
  in the frame $e_0, e_1, e_2$. Since the matrix representation of the
  Ricci tensor becomes too involved, we rather give its components
  with respect the same frame:
  \begin{align*}
    \Ric_{00} & = \frac{c_{1}^{4} - \mu^{2} - 2 \, \mu + 3}{2 \,
                {\left(\mu - 1\right)} \nu} = \frac{(c - 1)^2 - (\mu +
                1)^2 + 4}{2 \, \left(\mu - 1\right) \nu}, \\
    \Ric_{01} & =  \frac{c_{1}^{4} + 2 \, c_{1}^{2} \mu - 2 \,
                c_{1}^{2} - 3 \, \mu^{2} + 2 \, \mu + 1}{2 \,
                {\left(\mu - 1\right)} \nu} = \frac{(c + \mu)^2 - 4(\mu^2 + t - 1)}{2 \,
                {\left(\mu - 1\right)} \nu}, \\ 
    \Ric_{02} & = 0, \\
    \Ric_{11} & = -\frac{c_{1}^{4} \mu - 2 \, c_{1}^{4} - 4 \,
                c_{1}^{2} \mu - \mu^{3} + 4 \, c_{1}^{2} + 12 \,
                \mu^{2} - 17 \, \mu + 6}{2 \, {\left(\mu - 1\right)}
                \nu} \\
              & = \frac{ {\left(2 - \mu \right)} c^{2} + {\left(6
                \, \mu - 8\right)} c + \mu(\mu^{2} - 12 \,
                \mu + 12)}{2 \, {\left(\mu - 1\right)} \nu},
    \\ 
    \Ric_{12} & = 0, \\
    \Ric_{22} & =  -\frac{c_{1}^{4} - 2 \, c_{1}^{2} \mu + 2 \, c_{1}^{2} +
        \mu^{2} + 2 \, \mu - 3}{2 \, {\left(\mu - 1\right)}} =
                -\frac{(c - \mu)^2 + 4(\mu - 1)}{2(\mu - 1)} .
  \end{align*}

  Let $A$ be defined as in (\ref{eq:10}). The condition $A \in
  \so(\mathfrak g_c, \Ric)$, or equivalently
  \begin{equation*}
    A^T \Ric + \Ric A = 0
  \end{equation*}
  is given by the equations
  \begin{equation*}
    \frac{2 \, a_{11} {\left(c - \mu\right)}}{\nu} =
    \frac{{\left(a_{20}(c - 2) + a_{21}\right)} {\left(\mu -
          c\right)}}{\mu - 1} =  \frac{{\left( a_{20}(2 \mu  +
          c - 4 ) +  a_{21} (2 - \mu) \right)}
      {\left(\mu - c\right)}}{\mu - 1} = 0.
  \end{equation*}

  Notice that if $\mu \neq c$, then the only possible solution is $A =
  0$. In this case, we use the same ideas from previous cases to prove
  $\I(G_c, g_{\mu, \nu}) \simeq G_c$.

  Finally, assume that $\mu = c$. In this case, one can see that $\Ric
  = - \frac{2}{\nu} \, g_{c, \nu}$ and so the metric on $G_c$ is
  Einstein. Moreover, $(G_c, g_{c, \nu})$ is isometric to the
  hyperbolic space of curvature $-\frac{1}{\nu}$ and hence its
  isometry group is isomorphic to $\SO(3, 1)$.
\end{proof}

\begin{remark}
  \label{sec:case-g_c-1-1}
  Putting together Theorems \ref{sec:case-g_i} and
  \ref{sec:case-g_c-1} one can construct and infinite family of pairs
  of solvable Lie groups endowed with a left invariant metric, say
  $(S, g)$ and $(S', g')$, such that $(S, g)$ is isometric to
  $(S', g')$ but $S$ is not isomorphic to $S'$. In fact, if $0 < \nu$
  and $1 < c$, then $(G_I, g_\nu)$ and $(G_c, g_{c, \nu})$ are both
  isometric to the $3$-dimensional hyperbolic space of curvature
  $-\frac{1}{\nu}$.
\end{remark}

\section{Computation of the index of symmetry}

With the results obtained in Section \ref{sec:full-isometry-groups},
we can easily compute the index of symmetry for all the non-unimodular
Lie groups of dimension $3$. Let us explain the general procedure to
compute this invariant. Let $(G, g)$ be a simply connected,
non-unimodular $3$-dimensional Lie group endowed with a left invariant
metric $g$. We may assume that $(G, g)$ is not a symmetric space,
otherwise the index of symmetry is equal to $3$. It follows from
Subsection~\ref{sec:index-symm-homog} that
$i_{\mathfrak s}(G, g) \le 1$ and we have that
$i_{\mathfrak s}(G, g) = 1$ if and only if there exists a Killing
vector field $X$ on $G$ such that $X_e \neq 0$ and $(\nabla X)_e = 0$,
where $e$ is the identity element of $G$. Now, if $\dim \I(G, g) = 3$
such $X$ will be a right invariant field. In the case where
$\dim \I(G, g) = 4$, the Lie algebra of the isotropy group is
generated by a single element, say
$A \in \so(T_eG) \simeq \so(\mathfrak g, g)$ and so we need to find a
right invariant vector field, $X \neq 0$ such that
$(\nabla X)_e = \alpha \, A$ for some $\alpha \in \mathbb R$. Recall
that the second situation only occurs when $G$ is isomorphic to $G_0$
and $(G, g)$ is isometric to $(G_0, g_{\mu, \nu})$ (see Theorem
\ref{sec:case-g_0}) and in this case we can take
\begin{equation}\label{eq:11}
  A =
  \begin{pmatrix}
    0 & 0 & -\nu \\
    0 & 0 & -\frac{2 \nu}{\mu} \\
    1 & 2 & 0
  \end{pmatrix}
\end{equation}
with respect to the basis $e_0, e_1, e_2$. After some standard
calculations we obtain the index of symmetry for every metric
described in Table \ref{tab:left-invariant-metrics}. We summarize the
results in Table~\ref{tab:index-of-symmetry}. It is worth mention that
the fourth column of Table \ref{tab:index-of-symmetry} collects the
vector fields generating the distribution of symmetry $\mathfrak s_g$,
and since $\mathfrak s_g$ is invariant by isometries, such fields can
be taken left invariant. For example, when $c = 0$ and
$g = g_{\mu, \nu}$, the distribution of symmetry is generated by
$e_0 - \frac12 e_1$ because there is a Killing vector field $X$ such
that $X_e = (e_0 - \frac 12 e_1)|_e$ and $(\nabla X)_e = 0$. Notice
that $X$ is not right invariant. In fact, in order to construct such
an $X$ one shows thar $(\nabla X)_e$ is a multiple of the
skew-symmetric endomorphism given by~(\ref{eq:11}).

\begin{table}[ht]
  \caption{The index and distribution of symmetry
    $i_{\mathfrak s}(G, g)$ and $\mathfrak s_g$ of the Lie group $G$
    with respect to the left invariant metric $g$.}
  \centering
  {\tabulinesep=1.2mm
    \begin{tabu}{|c|c|c|c|c|c|c|}
      \hline
      \multicolumn{2}{|c|}{$G$} & \multicolumn{3}{c|}{$g$} & $i_{\mathfrak s}(G, g)$ & $\mathfrak s_g$ \\ \hline \hline
      \multicolumn{2}{|c|}{$G_I$} & \multicolumn{3}{c|}{$g_\nu$} & $3$ & $TG_I$ \\ \hline
      \multirow{13}{*}{$G_c$} & \multirow{2}{*}{$c<0$} & \multirow{2}{*}{$g_{\mu,\nu}$} & \multicolumn{2}{c|}{$\mu < |c|$} & 0 & - \\ \cline{4-7}
      & & & \multicolumn{2}{c|}{$\mu=|c|$} & $1$ & $\mathbb R e_2$\\ \cline{2-7}
      & \multirow{2}{*}{$c=0$} & \multicolumn{3}{c|}{$g_{\mu,\nu}$} & $1$ & $\mathbb R(e_0-\frac12e_1)$\\ \cline{3-7}
      & & \multicolumn{3}{c|}{$g_\nu$} & 3 & $TG_0$ \\ \cline{2-7}
      & \multirow{3}{*}[-.5pc]{$0<c<1$}& \multirow{3}{*}{$g_{\mu,\nu}$} & \multicolumn{2}{c|}{$\mu = 0$} & 1 & $\mathbb R e_2$ \\ \cline{4-7}
      & & & \multicolumn{2}{c|}{$0<\mu\neq\sqrt c$} & 0 & - \\ \cline{4-7}
      & & & \multicolumn{2}{c|}{$\mu=\sqrt c$} & 1 & $\mathbb R(e_0+\frac1{\sqrt c}e_1)$ \\ \cline{2-7}
      & \multirow{3}{*}{$c=1$}& \multirow{2}{*}{$g_{\mu,\nu}$} & \multicolumn{2}{c|}{$\mu<1$} & 0 & - \\ \cline{4-7}
      & & & \multicolumn{2}{c|}{$\mu=1$} & 1 & $\mathbb Re_0$ \\ \cline{3-7}
      & & \multicolumn{3}{c|}{$g'_{\lambda,\nu}$} & 0 & - \\ \cline{2-7}
      & \multirow{3}{*}{$1<c$} & \multirow{3}{*}{$g_{\mu,\nu}$} & \multirow{2}{*}{$\mu<c$} & $\mu\neq(\sqrt c -1)^2+1$ & 0 & - \\ \cline{5-7}
      & & & & $\mu=(\sqrt c-1)^2+1$ & $1$ & $\mathbb R\left(\frac{c - 3 \sqrt c + 2}{\sqrt c - 1}e_0+e_1\right)$ \\ \cline{4-7}
      & & & \multicolumn{2}{c|}{$\mu=c$} & 3 & $TG_c$ \\ \hline
    \end{tabu}
  }
  \label{tab:index-of-symmetry}
\end{table}

\subsection{The geometric meaning of the subset of metrics with
  maximal index of symmetry inside the moduli space of left invariant
  metrics}

Let $G$ be a $3$-dimensional non-unimodular Lie group and denote by
$\mathcal M(G)$ the moduli space of left invariant metrics up to
isometric automorphism. Using the classification of Ha-Lee given in
Table \ref{tab:left-invariant-metrics}, we can identify
$\mathcal M(G)$ with a topological subspace of the symmetric space
$\Sym_3^+ = \GL_3(\mathbb R) / \OO(3)$ of positive definite inner
products on $\mathbb R^3$, and thus $\mathcal M(G)$ inherits this
natural topology. More precisely, we have the following homeomorphisms
\begin{equation*}
  \mathcal M(G) \simeq
  \begin{cases}
    \mathbb R^+, & G = G_I, \\
    (0, |c|] \times \mathbb R^+, & G = G_c \text{ with } c < 0, \\
    (\mathbb R^+ \times \mathbb R^+) \sqcup \mathbb R^+, & G = G_0,\\
    [0, 1) \times \mathbb R^+, & G = G_c \text{ with } 0 < c < 1, \\
    \mathcal X, & G = G_1, \\
    (1, c] \times \mathbb R^+, & G = G_c \text{ with } 1 < c,
  \end{cases}
\end{equation*}
where $\mathcal X$ can be thought as the gluing of the perpendicular
half-planes
\begin{equation*}
  \{(\mu, 0, \nu): \mu \in (0, 1], \, \nu \in \mathbb R^+\} \cup_{\mathcal L} \{(1, \lambda, \nu): \lambda \in [0, 1), \, \nu \in \mathbb R^+\}
\end{equation*}
along the line $\mathcal L = \{(1, 0, \nu): \nu \in \mathbb
R^+\}$. Notice that $g'_{\lambda, \nu}$ in Table
\ref{tab:left-invariant-metrics} makes sense for $\lambda = 0$ if one
defines $g'_{0, \nu} = g_{1, \nu}$. Recall that $\mathcal X$ is
homeomorphic to an open set of $\mathbb R^2$, but the natural
inclusion of $\mathcal X \subset \mathbb{R}^3$ into $\mathbb R^3$ is
not differentiable along the gluing line.

Now let $\mathcal S(G) \subset \mathcal M(G)$ be the subset of
(equivalence classes of) metrics with maximal index of symmetry and
let us denote by $\mathcal Z(G)$ the set of singularities of
$\mathcal M(G)$.  Recall that $\mathcal Z(G_I) = \varnothing$;
$\mathcal Z(G_c) \simeq \{|c|\} \times \mathbb R^+$ if $c < 0$;
$\mathcal Z(G_0) = \mathbb R^+$;
$\mathcal Z(G_c) = \{0\} \times \mathbb R^+ \cup \{\sqrt c\} \times
\mathbb R^+$ if $0 < c < 1$; $\mathcal Z(G_1) = \mathcal L$; and
$\mathcal Z(G_c) = \{c\} \times \mathbb R^+$ if $1 < c$.

We can prove the following result by direct inspection using Table
\ref{tab:index-of-symmetry}.

\begin{theorem}\label{sec:geom-mean-subs}
  Let $G$ be a non-unimodular $3$-dimensional Lie group. Then
  \begin{equation*}
    \mathcal Z(G) \subset \mathcal S(G).
  \end{equation*}
  Moreover, equality holds for every $G$ such that $G \not\simeq G_I$ and
  $G \not\simeq G_c$ with $0 < c < 1$.
\end{theorem}

\begin{remark}
  Let $0 < c < 1$ and let $g$ be a left invariant metric on $G_c$ such
  that $(G_c, g)$ is isometric to $(G_c, g_{\mu, \nu})$. It is easy to
  see that
  \begin{equation*}
\scal g = - \frac{2 (4 - c - 3 \mu^2)}{(1 - \mu^2) \nu}.
\end{equation*}
In particular, for all the metrics of the form $g = g_{\sqrt c, \nu}$,
which are precisely the ones in $\mathcal S(G_c) - \mathcal Z(G_c)$,
the scalar curvature $\scal g = -8 / \nu$ does not depend on $c$.
\end{remark}

\bibliography{/home/silvio/Dropbox/math/bibtex/mybib.bib}
\bibliographystyle{amsalpha}

\end{document}